\documentclass[12pt]{article}

\usepackage[english]{babel}
\usepackage{upref,amsfonts,amsxtra,latexsym,color}
\usepackage{amssymb,amsmath,amsthm,a4wide,epsf,mathrsfs,verbatim,hyperref}
\usepackage[all]{xy}

\newcommand{\ccF}{\mathfrak F} 
\newcommand{\ccW}{\mathfrak W}

\DeclareMathOperator{\dKK}{d_{KK}}

\newcommand{\cstar}{$\mathrm{C}^*$}
\newcommand{\cst}{\mathrm{C}^*}

\newtheorem{theorem}{Theorem}[subsection]
\newtheorem{lemma}[theorem]{Lemma}
\newtheorem{corollary}[theorem]{Corollary}
\newtheorem{proposition}[theorem]{Proposition}

\theoremstyle{definition}

\numberwithin{equation}{section}
\numberwithin{theorem}{section}

\newcommand{\dist}{\mathrm{dist}}

\newcommand{\sr}{{\rm sr}}

\newcommand{\ep}{{\varepsilon}}

\newcommand{\cU}{{\cal U}}

\newcommand{\cS}{{\cal S}}

\newcommand{\N}{{\mathbb N}}

\newcommand{\R}{{\mathbb R}}
\newcommand{\Cs}{{\cstar-al\-ge\-bra}}

\title{Axiomatizability of the stable rank of \Cs s} 
\author{Ilijas Farah\thanks{Partially supported by NSERC}
\ and Mikael R\o rdam\thanks{Supported  by the Danish National Research Foundation (DNRF) through the Centre for Symmetry and Deformation at University of Copenhagen, and by The Danish Council for Independent Research, Natural Sciences.}}
\date{\today} 

\begin{document} 
\maketitle

\begin{abstract}  We show that the class of \Cs s with stable rank greater than a given positive integer 
is axiomatizable in logic of metric structures.  As a consequence we show that the stable rank is continuous with respect to forming ultrapowers of \Cs s, and that stable rank is Kadison--Kastler stable.
\end{abstract}

\noindent The notion of stable rank of a \Cs{} was invented by Rieffel in \cite{Rie:Dimension} for the purpose of obtaining results on the non-stable $K$-theory of the \Cs. The stable rank of a \Cs, which can attain values in $\{1,2, \dots, \infty\}$, can be viewed as a ``non-commutative'' dimension of the \Cs; the stable rank of a \Cs{} is one if and only if the invertible elements (in its unitization) are dense. 

It was shown in \cite[\S 3.8]{Muenster} that the class of \Cs s with stable rank equal to one is axiomatizable in logic of metric structures. In this note we improve this result, as stated in the abstract, to all values of the stable rank. 
In \cite[Theorem~6.3]{christensen2010perturbations} it was proved that  if the Kadison-Kastler distance between two \Cs s is less than $1/8$, then if one of the two \Cs s has real rank zero, then so has the other. 
In  \cite[Question~7.3]{christensen2010perturbations} it was asked 
whether higher values of the real rank are 
stable under small perturbations and, likewise, what happens to the stable rank under small perturbations?
Theorem \ref{th:axiomatizable} (ii) provides  the answer to the second part of this question; 
see \cite[\S 3.9 and \S 5.15]{Muenster} for more information on the first part of the question,  stability of higher values of the 
real rank under small perturbations. 

It was shown in \cite{Ror-1988} that if $A$ is a unital \Cs{} with $\sr(A) > 1$, so that the invertible elements in $A$ are not dense, then there is an element $a \in A$ of norm 1 such that the distance from $a$ to the invertible elements of $A$ is equal to 1. We extend this result to the situation where $A$ is a unital \Cs{} with $\sr(A) > n$, for any integer $n \ge 1$ (Corollary~\ref{cor:max-distance}). This is the ingredient in the proof our main result, Theorem~\ref{th:axiomatizable}.

The results in \cite{Ror-1988} were extended by Brown and Pedersen in  \cite{BroPed-1995} to the situation of so-called ``quasi-invertible'' elements. We use the methods from \cite{BroPed-1995} to obtain Corollary~\ref{cor:max-distance}.
 
First we recall the definition of higher stable ranks from Rieffel's paper \cite{Rie:Dimension}. 
 Let $A$ be a unital \Cs. Let $\lg_n(A)$ be the set of $n$-tuples $a=(a_1, \dots, a_n) \in A^n$ for which $A = Aa_1+\cdots + Aa_n$. Identify $A^n$ with $M_{n,1}(A)$ equipped with the usual norm inherited from $M_n(A)$. 
   With this convention, $A^n$ is an $M_n(A)$--$A$ bimodule, and if $a \in A^n$, then $a^* \in M_{1,n}(A)$ (see e.g.\ \cite{lance1995hilbert} for the 
   theory of Hilbert \cstar-modules).  Rieffel defined $\sr(A)$ to be the smallest integer $n \ge 1$ such that $\lg_n(A)$ is dense in $A^n$ (and if no such integer exists, then $\sr(A) = \infty$); 
   see \cite{Rie:Dimension} where stable rank was introduced as the `topological stable rank.'

If $a,b \in A^n$, then $a^*b \in A$ and $ab^* \in M_n(A)$. Note that $a \in \lg_n(A)$ if and only if $a^*a$ is invertible in $A$. In that case, $a = v|a|$, where $|a| = (a^*a)^{1/2} \in A$ and $v = a|a|^{-1} \in A^n$ is an isometry, i.e., $v^*v=1$. Note also 
that $\|a\|^2 = \|a^*a\|=\|\sum_{i\leq n} a_i^*a_i\|$.

We will now assume that $A$ is represented faithfully and non-degenerately on a Hilbert space $H$, so that $A \subseteq B(H)$. Then $A^n \subseteq B(H,H^n)$. Thus each $a \in A^n$ admits a polar decomposition $a = v|a|=|a^*|v$ with $|a|$ as above and with $v$ a partial isometry in $B(H,H^n)$. This polar decomposition agrees with the one above when $a \in \lg_n(A)$, but in general $v$ is not an element of $A^n$.
However, if  $\psi\in C([0,1])$ is such that $\psi(0)=0$ and $a = (a_1, \dots, a_n) \in A^n$ has polar decomposition 
$a=v|a|$, then $v\psi(|a|)$ is a limit of $v P_n(|a|)$ for a sequence of polynomials 
$\{P_n\}$ vanishing at 0. Therefore the $j$th entry of $v\psi(|a|)$ belongs to $\cst(a_j, |a|)$, $1 \le j \le n$, and in particular, $v\psi(|a|)$ belongs to $A^n$. This standard fact will be used several times
below.

The lemma below is an analogue of \cite[Proposition 1.7]{BroPed-1995}.

\begin{lemma} \label{lm:Prop1.7} Let $a \in A^n$, $b \in \lg_n(A)$, and $\beta > \|a-b\|$ be given. Write $b = w|b|$, with $w$ an isometry in $A^n$. It follows that $a+\beta w \in \lg_n(A)$. 
\end{lemma}

\begin{proof} 
Since $|b|+\beta \cdot 1$ is bounded below by $\beta \cdot 1$ it is invertible with $\|(|b|+\beta \cdot 1)^{-1}\| \le \beta^{-1}$. Next, 
 $c = w^*(a-b) = w^*a-|b|$ satisfies $\|c\|< \beta$, so $\|c(|b|+\beta 1)^{-1}\|  < 1$, which implies that $c(|b|+\beta)^{-1}+1$ is invertible, being at distance less than $1$ from the identity. 
Therefore  
\[
w^*a+\beta \cdot 1 = c+|b|+\beta \cdot 1 = \big(c(|b|+\beta \cdot 1)^{-1}+1\big)(|b|+\beta \cdot 1)
\]
is an invertible element of $A$. Now, $a+\beta w = w(w^*a+\beta \cdot 1)$ is the product of an element of $\lg_n(A)$ and an invertible element of $A$, and therefore 
belongs to $\lg_n(A)$.
\end{proof}

\noindent
The lemma below, and its proof, closely resembles \cite[Theorem 2.2]{BroPed-1995}, which again was a refinement of \cite[Lemma 2.1]{Ror-1988}. Let us fix some notation. Let $a \in A^n$ with polar decomposition $a = v|a|$ be given (with $v$ a partial isometry in $B(H,H^n)$). For each $\lambda \ge 0$, let $e_\lambda \in B(H)$ and $f_\lambda \in B(H^n)$ denote the spectral projections corresponding to the interval $[0,\lambda]$ of $|a|$ and $|a^*|$, respectively, i.e., $e_\lambda = 1_{[0,\lambda]}(|a|)$ and $f_\lambda = 1_{[0,\lambda]}(|a^*|)$.

\begin{lemma} \label{lm:polar1}
Let $n \ge 1$ be an integer, and let $a \in A^n$ with polar decomposition $a = v|a|$ be given. Let $f_\lambda$ be as above, for $\lambda \ge 0$. Then, for each $\gamma > \dist(a,\lg_n(A))$, there exists $s \in \lg_n(A)$ such that $(1-f_\gamma)v = (1-f_\gamma)s$. 
\end{lemma}

\begin{proof} 
Choose $\dist(a,\lg_n(A))< \beta <  \gamma$.  Find $b \in \lg_n(A)$ such that $\|a-b\| < \beta$, and write $b = w|b|=|b^*|w$, with $w$ an isometry in $A^n$.  Let $\varphi,\psi \colon \R^+ \to \R^+$ be the continuous functions given by 
\[
\varphi(t) = \min\{\gamma^{-1},t^{-1}\}, \qquad \psi(t) = \min\{t\gamma^{-2},t^{-1}\}, \quad t \in \R^+. 
\]
Then $\psi(|a|)v^* = (v\psi(|a|))^* \in M_{1,n}(A)$ because $\psi(0)=0$; and $\|\beta\psi(|a|)v^*w\| < 1$ because $\|\beta \psi\|_\infty =\beta \gamma^{-1}< 1$. 
Since $\varphi(|a|)$ is bounded below by $\|a\|^{-1} \cdot 1$ and therefore invertible, 
it follows from Lemma~\ref{lm:Prop1.7} that 
\[
s = (a + \beta w)\big(1+\beta\psi(|a|)v^*w\big)^{-1} \varphi(|a|)
\]
belongs to $\lg_n(A)$.

With $e_\lambda$ and $f_\lambda$, $\lambda \ge 0$, as defined as above, note that  
 $v(1-e_\gamma)=(1-f_\gamma)v$, 
 $v(1-e_\gamma)v^*=(1-f_\gamma)$, 
  and 
\[
 (1-e_\gamma)|a|\varphi(a)=(1-e_\gamma)|a|\psi(a)=(1-e_\gamma).
\] 
 Hence,
$$(1-f_{\gamma})(a+\beta w)= v(1-e_{\gamma})|a|+\beta v(1-e_{\gamma})v^*w = v(1-e_{\gamma})|a|\big(1+\beta\psi(|a|)v^*w\big),$$
so 
\[
(1-f_{\gamma})s= v(1-e_{\gamma})|a|\varphi(|a|) = v(1-e_{\gamma}) = (1-f_{\gamma})v,
\]
as required. 
\end{proof}

\noindent As in \cite[Theorem 2.2]{BroPed-1995} and \cite[Theorem 2.2]{Ror-1988} one can improve the lemma above as follows: if $a \in A^n$ and $\gamma$ are as in Lemma~\ref{lm:polar1}, then there exists an isometry  $u \in A^n$ such that $v(1-e_\gamma) = u(1-e_\gamma)$.  However, we shall not need this stronger statement to obtain Corollary \ref{cor:max-distance} and Theorem~\ref{th:axiomatizable}  below.

Recall that for a positive element $a \in A$ and $\lambda \ge 0$ we define $(a-\lambda)_+ = g(a)$, where $g(t) = \max\{t-\lambda,0\}$.

\begin{proposition} \label{prop:distance}
Let $n \ge 1$ be an integer, and let $a \in A^n$ with polar decomposition $a = v|a|$ be given. Then
$$
\dist(a,\lg_n(A)) = \inf \{\lambda  \ge 0 : \exists \; s \in \lg_n(A) \; \text{such that} \; (1-f_\lambda)v = (1-f_\lambda)s\}. 
$$
\end{proposition}

\begin{proof} 
The inequality ``$\ge$'' follows from Lemma~\ref{lm:polar1}. To prove the other inequality, take $\lambda \ge 0$ for which there exists $s \in \lg_n(A)$ with $(1-f_\lambda)v = (1-f_\lambda)s$. For any continuous function $h \colon \R^+ \to \R^+$ which vanishes on $[0,\lambda]$ we have  $h(1-1_{[0,\lambda]})  = h$. Hence 
\[
h(|a^*|)v = h(|a^*|)(1-f_\lambda)v = h(|a^*|)(1-f_\lambda)s = h(|a^*|)s.
\]
In particular, $(|a^*|-\lambda)_+v = (|a^*|-\lambda)_+s$,
and  this element belongs to the closure of $\lg_n(A)$. Indeed, if $s \in \lg_n(A)$ and $c$ is any positive element of $M_n(A)$, then $cs$ belongs to the closure of $\lg_n(A)$, because $c+\ep \cdot 1$ is invertible in $M_n(A)$ for all $\ep >0$, whence $s(c+\ep \cdot 1)$ belongs to $\lg_n(A)$.
 This shows that
\begin{eqnarray*} 
\dist(a, \lg_n(A)) & \le & \|a - (|a^*|-\lambda)_+v\| \\
&=& \|(|a^*| - (|a^*|-\lambda)_+)v)\| \; \le \; \||a^*| - (|a^*|-\lambda)_+\| \;  \le \; \lambda,
\end{eqnarray*} 
as required. 
\end{proof}

\begin{corollary} \label{cor:max-distance} Let $A$ be a unital \Cs{} with $\mathrm{sr}(A) > n$, where $n \ge 1$ is an integer. It follows that there exists $b \in A^n$ such that 
$$\|b\| = \dist(b,\lg_n(A)) = 1.$$
\end{corollary}

\begin{proof} Take any $a \in A^n$ with $\gamma = \dist(a,\lg_n(A)) > 0$. As above, write $a = v|a|=|a^*|v$ with $v$ a partial isometry in $B(H,H^n)$. Consider the continuous function $h \colon \R^+ \to \R^+$ given by $h(t) = \min\{\gamma^{-1}t,1\}$. Set $b = vh(|a|)=h(|a^*|)v$. Then $b \in A^n$ because $h(0) = 0$, and $|b^*| = h(|a^*|)$. Also, $\|b\| = \|h(|a|))\| \le 1$. Let $f_\lambda$ and $\widetilde{f}_\lambda$ be the spectral projections corresponding to the interval $[0,\lambda]$ for $|a^*|$ and $|b^*| = h(|a^*|)$, respectively. 

Suppose that $\alpha= \dist(b,\lg_n(A)) < 1$. Then, for any $\alpha < \beta <1$, by Lemma~\ref{lm:polar1}, there exists $s \in \lg_n(A)$ such that $(1-\widetilde{f}_\beta)v = (1-\widetilde{f}_\beta)s$. Since $1_{[0,\beta]} \circ h = 1_{[0,\gamma \beta]}$ we get that $\widetilde{f}_\beta =1_{[0,\beta]}(|b^*|) = (1_{[0,\beta]} \circ h)(|a^*|) = f_{\gamma \beta}$. Hence $(1-f_{\gamma \beta})v = (1-f_{\gamma \beta})s$. By Proposition~\ref{prop:distance} this would entail that $\dist(a,\lg_n(A)) \le \gamma \beta < \gamma$, a contradiction. Hence $\dist(b,\lg_n(A)) \ge 1$. This completes the proof, since $\dist(x,\lg_n(A)) \le \|x\|$ holds for all $x \in A^n$ (as $(0, \dots, 0)$ belongs to the closure of $\lg_n(A)$).
\end{proof}

\noindent
Prior to proving that $\sr(A)\geq n$ is axiomatizable for every $n\geq 1$ we recall  some facts and definitions 
from \cite[\S 2.1]{Muenster}. 
Fix $k\geq 1$. 
The space 
 of  formulas of the language of \cstar-algebras with free variables among $\bar x=(x_1,\dots, x_k)$ 
is denoted $\ccF^{\bar x}$
(\cite[Definition~2.1.1]{Muenster}). For $\varphi\in \ccF^{\bar x}$, a \cstar-algebra $A$, 
and a $k$-tuple $\bar a$ in $A$ 
 of the same sort as $\bar x$
(typically $\bar a$ is  a $k$-tuple in the unit ball of $A$), 
by $\varphi^A(\bar a)$ we denote the interpretation of $\varphi$
in $A$ at $\bar a$.   The space~$\ccF^{\bar x}$ has a natural algebra structure, and it is 
 equipped with  the  seminorm
\[
\| \varphi \|_T = \sup  \varphi^A(\bar a),
\]
  where $A$ ranges over all \cstar-algebras and $\bar a$ ranges over all $k$-tuples of
  the same sort as~$\bar x$. The completion of $\ccF^{\bar x}$ with respect to this norm is    a
      Banach algebra, denoted  $\ccW^{\bar x}$, 
       (\cite[\S 3.1(a)]{Muenster}).
       Its elements 
are called  \emph{definable predicates}. Hence a definable predicate is an assignment
of predicates (i.e., real-valued uniformly continuous functions on $A^k$)  of a particular form to \cstar-algebras. 
An assignment of closed subsets to \cstar-algebras $A\mapsto \cS^A\subseteq A^k$ 
is a \emph{definable set} (\cite[Definition~3.2.1]{Muenster}) 
if for any formula $\psi(\bar x,y)$ of the language of \cstar-algebras
both 
 $\sup_{y\in\cS^A} \psi(\bar x,y)^A$ 
 and $\inf_{y\in \cS^A} \psi(\bar x,y)^A$ 
are definable predicates (inconveniently, being a definable set 
is a bit stronger than being the zero-set of a definable predicate). 
An extension of a language of metric structures is \emph{conservative} if 
every model in the original language can be expanded uniquely to a model of the new language. 
In particular, an extension obtained by adding definable predicates and 
allowing quantification (i.e., taking sups and infs) 
over definable sets is conservative.

 \begin{lemma} \label{lm:definable} 
For any $n\geq 1$,  the extension of the language of \cstar-algebras 
 by allowing quantification over the unit ball 
$A^n_1$  of $A^n$ and adding a predicate $F_n$ for  the standard norm on~$A^n$ 
  is conservative. 
 \end{lemma} 

\begin{proof} It suffices to prove that $F_n$ 
is a definable predicate and the unit ball of $A^n$ is a definable set.  
 This proof is virtually identical to the proof that the operator norm on $M_n(A)$ is a definable predicate
 and the unit ball of $M_n(A)$ is a definable subset of $A^{n^2}$ (\cite[Lemma~4.2.3]{Muenster}). 

It is clear that function $F_n$ commutes with taking ultraproducts, i.e.,
if $A_j$, for $j\in J$, are \cstar-algebras, $\cU$ is an ultrafilter on $J$, and $A=\prod_{\cU} A_j$,  then 
for any $\bar a\in A$ and any representing sequence $(\bar a_j)$ of $\bar a\in A$ we have  
$F_n^A(\bar a)=\lim_{j\to \cU} F_n^{A_j}(\bar a_j)$. 
Therefore the class  ${\mathcal C}'$  of all structures of the form $(A,F_n^A)$, where $A$ is a \cstar-algebra, 
is closed under taking ultraproducts and ultraroots, and by \cite[Theorem~2.4.1]{Muenster}
axiomatizable.  
Beth Definability Theorem (\cite[Theorem~4.2.1]{Muenster}) now implies that $F_n$ is a definable predicate. 
Since every  element in  $A^n_1$  has a representing sequence in  $\prod_j (A_j)^n_1$, 
we have  $A^n_1=\prod_{\cU}(A_j)^n_1$. 
Therefore~$A^n_1$ is a definable set by  \cite[Theorem~3.2.5]{Muenster}. 
This completes the proof. 
\end{proof} 

We shall also need the fact that 
the  set $A_+$ of positive elements in a \cstar-algebra $A$ is definable (\cite[Example~3.2.6 (2)]{Muenster})
and therefore the language extended by allowing quantification over the positive part of the unit ball is
a conservative extension of the language of \cstar-algebras. 

\begin{lemma} \label{lm:phi} 
For $n\geq 1$, consider the sentence 
\[
\varphi_n=\sup_{x\in A^n, \|x\|\leq 1} \; 
\inf_{v\in A^n, \|v\|\leq 1} \; 
\inf_{y\in A_+, \|y\|\leq 1}
\max(\|x-vy\|, \|v^*v-1\|)
\]
in the extended language of unital \cstar-algebras. 
For every unital \cstar-algebra $A$ we have~$\sr(A)\leq n$ if and only if $\varphi_n^A=0$, and 
$\sr(A)>n$ if and only if $\varphi_n^A\geq 1$. 
\end{lemma} 

\begin{proof} By a standard geometric series argument,    $\|v^*v-1\|<1$ implies that $v^*v$ 
is invertible and hence that
$v=(v_1,\dots, v_n)$ is in~$\lg_n(A)$.  
It was proved in   \cite[Lemma~3.8.4]{Muenster} 
that  $a=(a_1,\dots, a_n)\in A^n$ such that $\max_{i\leq n} \|a_i\|\leq 1$ 
 is in the closure of~$\lg_n(A)$ if and only if it  can be approximated arbitrarily well 
by elements of the form $vy$, for $v\in A^n$ and $y\in A_+$,
such that $\|v^*v-1\|<1$
 and  $\|y\|\leq \sqrt n$. After renormalization of $A^n$,  we have   
 that $a\in A^n_1$  is in the closure of~$\lg_n(A)$ if and only if it  can be approximated arbitrarily well 
by elements of the form $vy$, for $v\in A^n$ and $y\in A_+$,
such that $\|v^*v-1\|<1$ and  $\|y\|\leq 1$. 
Therefore  $\varphi_n^A=0$ is equivalent to    $\sr(A)\leq n$, and   consequently 
  $\varphi_n^A\geq 1$ implies $\sr(A)>n$. 
  
   It  remains to prove that $\sr(A)>n$ implies 
$\varphi_n^A\geq 1$.   
Suppose otherwise, that $\sr(A)>n$ and $\varphi_n^A<1$. 
Let $b\in A^n$ be as in Corollary~\ref{cor:max-distance}, so that 
$\|b\| = \dist(b,\lg_n(A)) = 1$. Since $\varphi_n^A<1$, there exist $v\in A^n$ and $y\in A_+$ such that $\|y\|\leq 1$,  
$\|b-vy\|<1$, and $\|v^*v-1\|<1$. The latter formula implies $v\in \lg_n(A)$. 
With $\ep>0$ small enough to have $\|b-v(y+\ep \cdot 1)\|<1$, we have $v (y+\ep \cdot 1)\in \lg_n(A)$, 
contradicting the choice of $b$. 
\end{proof} 

\noindent Recall that the  \emph{Kadison--Kastler distance} between subalgebras $A$ and $B$ of $B(H)$ for a fixed Hilbert space $H$ is defined as  
\[
\dKK(A,B)=\max(\sup_{x\in A_1}\inf_{y\in B_1} \|x-y\|, \sup_{y\in B_1} \inf_{x\in A_1} \|x-y\|),
\]
writing $A_1$ and $B_1$ for the unit balls of $A$ and $B$, respectively. Thus $\dKK(A,B)$ 
is equal to the Hausdorff distance between $A_1$ and $B_1$.

\begin{theorem} \label{th:axiomatizable} For every integer $n\geq 1$ the classes of \cstar-algebras with 
stable rank greater than $n$, less than or equal to $n$, and equal to $n$, respectively,  are axiomatizable in logic of metric structures. 
In particular:
\begin{enumerate}
\item\label{I.ultra}  If $A_m$, for $m\in \N$, are unital \cstar-algebras and $\cU$ is an ultrafilter on $\N$, then 
\[
\sr\big(\prod_{\cU} A_m\big)=\lim_{n\to \cU} \sr(A_n).  
\]
In particular, stable rank is preserved under ultrapowers. 
\item\label{I.KK.1}  For every $n\geq 1$,
there exists $\ep_n>0$ such that for  any two unital \cstar-subalgebras  $A$ and $B$ of $B(H)$ with 
$\dKK(A,B)<\ep_n$, either $\sr(A)=\sr(B)$ or both of $\sr(A)$ and $\sr(B)$ are greater than $n$.
\end{enumerate}
\end{theorem} 

\begin{proof} In \cite[Proposition~3.8.1]{Muenster} it was proved that having stable rank at most $n$ is axiomatizable. 
By Lemma~\ref{lm:phi} and Lemma~\ref{lm:definable} 
 having stable rank greater than $n$ is also axiomatizable, and the axiomatizability of stable rank being equal to $n$ 
 follows.  

By \L o\'s's Theorem (\cite[Theorem~2.3.1]{Muenster}) with $\varphi_n$ as in Lemma~\ref{lm:phi} 
we have $\varphi_n ^{\prod_{\cU} A_m}=\lim_{m\to \cU} \varphi_n^{A_m}$, 
and therefore \ref{I.ultra} follows. We infer from 
\cite[Lemma~5.15.1]{Muenster} that  the evaluation of $\varphi_n$ 
is continuous with respect to $\dKK$ for every integer $n \ge 1$. Choose $\ep_n>0$ small enough so 
that $\dKK(A,B)<\ep_n$ implies $\|\varphi_j^A-\varphi_j^B\|<1$ for all $j\leq n$ and for all unital \cstar-subalgebras  $A$ and $B$ of $B(H)$. 
Then $\sr(A)=j<n$ implies $\sr(B)=j$, as required.  
\end{proof}

{\small
{
\bibliographystyle{amsplain}
\bibliography{stablerank} 

\providecommand{\bysame}{\leavevmode\hbox to3em{\hrulefill}\thinspace}
\providecommand{\MR}{\relax\ifhmode\unskip\space\fi MR }
\providecommand{\MRhref}[2]{%
  \href{http://www.ams.org/mathscinet-getitem?mr=#1}{#2}
}
\providecommand{\href}[2]{#2}
\begin{thebibliography}{1}

\bibitem{BroPed-1995}
L.~G. Brown and G.~K. Pedersen, \emph{On the geometry of the unit ball of a
  \cstar-algebra}, J.\ Reine Angew.\ Math. \textbf{469} (1995), 113--147.

\bibitem{christensen2010perturbations}
E.~Christensen, A.~M.\ Sinclair, R.~R.\ Smith, and S.~A.\ White,
  \emph{Perturbations of {\cstar}-algebraic invariants}, Geom. Funct. Anal.
  \textbf{20} (2010), no.~2, 368--397.

\bibitem{Muenster}
I.~Farah, B.~Hart, M.~Lupini, L.~Robert, A.~Tikuisis, A.~Vignati, and
  W.~Winter, \emph{Model theory of nuclear \cstar-algebras}, arXiv preprint
  arXiv:1602.08072 (2016).

\bibitem{lance1995hilbert}
E.~C. Lance, \emph{Hilbert \cstar-modules: a toolkit for operator algebraists},
  London Mathematical Society Lecture Note Series, vol. 210, Cambridge
  University Press, 1995.

\bibitem{Rie:Dimension}
M.~A Rieffel, \emph{Dimension and stable rank in the {K}-theory of
  \cstar-algebras}, Proc. London Math. Soc \textbf{46} (1983), no.~3, 301--333.

\bibitem{Ror-1988}
M.~R{\o}rdam, \emph{Advances in the theory of unitary rank and regular
  approximation}, Ann. of Math. \textbf{128} (1988), 153--172.

\end{thebibliography}

%
%
%
%
%

\end{document}